\documentclass[letterpaper, 10 pt, conference]{ieeeconf}

\IEEEoverridecommandlockouts                              

\usepackage[linesnumbered,ruled]{algorithm2e}
\usepackage{amsmath} 
\usepackage{amssymb}  
\usepackage{mathtools}
\usepackage{bm}
\usepackage[hidelinks]{hyperref}
\usepackage{cleveref}
\usepackage{import}
\usepackage{xcolor}
\usepackage{comment}
\usepackage[font=small]{caption}
\usepackage{graphicx} 
\usepackage[font=scriptsize]{subfig}
\usepackage{multicol}
\usepackage{siunitx}

\newtheorem{theorem}{Theorem}

\newtheorem{lemma}{Lemma}
\newtheorem{remark}{Remark}

\newtheorem{assumption}{Assumption}
\newtheorem{corollary}{Corollary}

\usepackage[maxbibnames=10,giveninits=true,sorting=none]{biblatex} 
\AtEveryBibitem{%
  \clearfield{issn} 
  \clearfield{doi} 
  \clearfield{url}
   \clearfield{number}
   \clearfield{organization}
}
\DeclareFieldFormat*{year}{{}}
\renewbibmacro{in:}{} 

\newcommand{\tr}[1]{{#1}^{\ensuremath{\mathsf{T}}}} 

\usepackage{xcolor}

\title{An Optimal Solution to Infinite Horizon Nonholonomic and Discounted Nonlinear Control Problems }

\author{Mohamed Naveed Gul Mohamed, Abhijeet, Aayushman Sharma, Raman Goyal, Suman Chakravorty%
\thanks{The authors are with the Department of Aerospace Engineering, Texas A\&M University, College Station, TX 77843, USA. \{\tt naveed, abhinir, aayushmansharma, schakrav\}@tamu.edu and ramaniitr.goyal92@gmail.com}
}

\begin{document}

\maketitle

\begin{abstract}
This paper considers the infinite horizon optimal control problem for nonlinear systems. Under the condition of nonlinear controllability of the system to any terminal set containing the origin and forward invariance of the terminal set, we establish a regularized solution approach consisting of a ``finite free final time" optimal transfer problem to the terminal set, which renders the set globally asymptotically stable. Further, we show that the approximations converge to the optimal infinite horizon cost as the size of the terminal set decreases to zero. We also perform the analysis for the discounted problem and show that the terminal set is asymptotically stable only for a subset of the state space and not globally. The theory is empirically evaluated on various nonholonomic robotic systems to show that the cost of our approximate problem converges and the transfer time into the terminal set is dependent on the initial state of the system, necessitating the free final time formulation. We also do comparisons of our free-final time approach with nonlinear MPC.
\end{abstract}
\begin{keywords}
Nonlinear control, Infinite horizon optimal control, Control Lyapunov function
\end{keywords}
\section{Introduction}
Optimal control methods are widely used for optimizing a performance index, subject to dynamic constraints. In many practical applications, it is desired to have an optimal control law that guarantees global asymptotic stability (GAS) for the closed-loop system.  Formulating an optimal control problem with an infinite horizon ensures that the resulting closed-loop system is globally asymptotically stable (GAS). This is due to the fact that different initial states typically require different times to reach the desired terminal condition, which is captured by the infinite horizon problem.
Finding an optimal control law for an infinite horizon problem subject to nonlinear dynamics is a challenging task \cite{bertsekas_vol1}. Due to the complexity of this challenge, the problem is restructured and tackled using a practical approach, incorporating a transfer from a nonlinear to a linear regime to simplify the solution process \cite{mohamed2023infinitehorizon}. This work is an extension of the approach to nonholonomic systems and systems that are not linearly controllable around the equilibrium. 

The solution to the stationary Hamilton-Jacobi-Bellman (HJB) equation can be used to compute the optimal feedback control law for continuous-time systems. Equivalently, Dynamic programming can be used for discrete-time optimal control problems \cite{bertsekas_vol1,Bellman:1957}. Obtaining a globally asymptotically stabilizing control by solving HJB, though appealing \cite{bernstein1993nonquadratic,wan1992family, wan1995nonlinear}, suffers the dreaded ``Curse of Dimensionality" \cite{bertsekas_vol1,Bellman:1957}. Thus, there is extensive literature on Approximate DP (ADP) and Reinforcement Learning that seeks to alleviate the curse of dimensionality. ADP methods \cite{ADP_handbook,lewis2013reinforcement} provide an approximately optimal policy/value function with high confidence. To overcome these shortcomings, `deep reinforcement learning' \cite{silver2016mastering1, schulman2015trust,TD3} has been widely used to approximate function using (deep) neural networks. However, these methods suffer from the curse of variance, and the training time could be unrealistically high \cite{henderson2018deep,D2C2.0_CDC,arxiv_D2C2.0}.

The field of Model Predictive Control (MPC) takes a ``direct approach" to solve the infinite horizon problem; however, owing to the infinite horizon of the involved optimal control problem, MPC computes the current control action by solving a ``fixed final time" finite horizon problem and repeats the process at the next state \cite{mayne2014model,mayne2000constrained}. Most MPC approaches then show the asymptotic stability of the resulting ``time invariant" control policy. There are two primary approaches:  the first is to use a suitable terminal cost function in the optimization problem that is a control Lyapunov function for the system in some terminal set containing the origin \cite{mayne2014model}.  The domain of attraction of the MPC law under this approach can be undesirably small and different methods have been suggested to increase the domain of attraction \cite{MPC_GAS1,MPC_GAS2,MPC_GAS3}. Alternatively, one can eschew the use of a terminal cost function and set using a suitable long horizon and well-designed incremental costs \cite[Ch.6]{grune2017}, but this typically leads to intractability owing to very long prediction horizons \cite{MPC_FTC1}. 
Additionally, most MPC approaches, like the quasi-infinite horizon approach \cite{denicolao1998quasi}, assume the system is controllable around the origin and gets a linear feedback law to control the system to the origin in the terminal set. In this regard, nonholonomic systems pose a special challenge since their linearization is uncontrollable around the origin or any desired state \cite{ROSENFELDER2023110972}. Also, the preferred choice of approximating the terminal cost with the cost-to-go of the linear controller is no longer possible. Hence, one has to control the system to the origin or in the neighborhood of the origin using a purely nonlinear controller.

Our approach is analogous to MPC in that we ``directly" solve the optimal control problem, but the key difference is that given an initial state, we solve a ``free-final time" problem for insertion into a terminal set. We used this approach to address the problem where the system linearization is controllable around the origin in previous work \cite{mohamed2023infinitehorizon}. In this work, we relax the linear controllability assumption to address the general problem and provide similar guarantees. We address the discounted infinite horizon problem owing to its wide use in the RL literature and show similar results as in the undiscounted case. Finally, there is no need for replanning in our approach owing to the free-final time. The primary limitation is that we do not consider state or control constraints in the problem as is typically done in MPC. 
The primary contribution of this paper is a tractable direct approach for the solution of infinite horizon optimal control problems that is globally asymptotically stabilizing for nonlinear systems under a mild nonlinear controllability assumption into a terminal set containing the origin. We also show that the approximation converges to the optimal infinite horizon cost as the size of the terminal set is reduced to zero. The rest of the paper is organized as follows: we introduce the problem in Section \ref{section:prob}, the solution approach for the undiscounted problem and the discounted problem are detailed in Section~\ref{section:sol_IHOCP} and \ref{section:sol_DIHOCP}, respectively, and the method is tested empirically on several nonlinear systems in Section~\ref{section:results}.

\section{Problem Formulation}\label{section:prob}
Let us consider the following infinite horizon optimal control problem (IH-OCP):
\begin{align}
    J_{\infty}^*(x) = \min_{\{u_k\}} \sum_{k=0}^{\infty} c(x_k, u_k);~~ &\text{given} ~x_0 = x \label{eq:IHOCP} \tag{IH-OCP}\\
    \text{subject to the dynamics: } x_{k+1} &= f(x_k, u_k), \label{eq:dynamics}
\end{align}
where $x_k \in \mathcal{X} \subset \mathbb{R}^n$ represents the state of the dynamical system, $u_k \in \mathcal{U} \subset \mathbb{R}^p$ represents the control input to the dynamical system, and $c(x_k, u_k)$ is the incremental cost incurred in taking control action $u_k$ at state $x_k$. The above problem is an infinite horizon optimal control problem, and thus, solving the problem is, in general, intractable owing to the infinite horizon of the problem. Our goal in this work is to develop a tractable approach to solve the above problem by transforming the problem into a suitable finite horizon problem.

Given that we can obtain a solution to the \eqref{eq:IHOCP}, it is well known that the infinite horizon cost-to-go $J^*_{\infty}(\cdot)$ satisfies Bellman's equation \cite[Ch.7]{bertsekas_vol1}:
\begin{align}
    J_{\infty}^*(x) = \min_{u} \{ c(x, u) + J_{\infty}^*(f(x, u)) \}. \label{eq:bellman}
\end{align}
We restate Corollary 1 from \cite{mohamed2023infinitehorizon} below for the sake of completeness.
\begin{corollary} \label{corollary.1}
    Let $J^*_\infty(x)$ satisfy the Bellman equation \eqref{eq:bellman}, then it is a control Lyapunov function for the system in \eqref{eq:dynamics} that renders the origin globally asymptotically stable.
\end{corollary}

Further, suppose that if there exists a $J_{\infty}(\cdot)$ such that it satisfies the Bellman equation (not necessarily optimal)
\begin{align}
    J_{\infty}(x) = \min_{u} \{ c(x, u) + J_{\infty}(f(x, u)) \}, \label{eq:suboptimal_bellman}
\end{align}
then $J_{\infty}(\cdot)$ also is a CLF that renders the origin globally asymptotically stable (GAS).

Thus, another goal for us in solving \eqref{eq:IHOCP} is to construct CLFs as in \eqref{eq:bellman} and \eqref{eq:suboptimal_bellman}, such that they render the origin GAS. In this work, we focus on the specific class of systems that are not linearly controllable around the origin, complimentary to the linearly controllable case considered in \cite{mohamed2023infinitehorizon}. Nonholonomic systems fall under this category. Though we cannot guarantee GAS of the origin, we aim to asymptotically stabilize the system into a terminal set. 

\section{Solution to the Infinite Horizon Optimal Control Problem}\label{section:sol_IHOCP}
The cost of IH-OCP can be written as
\begin{align*}
    J_{\infty}^*(x) = \min_{\{u_k\}}\Big[ \sum_{k=0}^{T - 1} c(x_k, u_k) +  \sum_{k=T}^{\infty} c(x_k, u_k) \Big],
\end{align*}
where we choose a $T$ such that the cost-to-go from $x_T$ - $J_{\infty}^*(x_T)$ - is very small compared to the cost to transfer from initial state $x$ to $x_T$. If one has knowledge of the cost-to-go function $J_{\infty}^*(\cdot)$ around the origin, i.e., the cost-to-go of the linearized system, one can pose the IH-OCP as finite horizon problem with $J_{\infty}^*(x_T)$ as an arbitrarily good approximation of the terminal cost. Since we do not have knowledge of the true cost-to-go function as the system linearization is uncontrollable, we instead use a heuristic cost-to-go function $\phi(x)$ and pose the finite horizon problem. Though it is a heuristic, we construct a formulation whose cost converges to the true IH-OCP cost in the limit.

Let us define the finite-horizon optimal control problem (FH-OCP):
\begin{align}
    J^T_\infty(x) &= \min_{\{u_k\}}  \sum_{k=0}^{T-1} c(x_k,u_k) +\phi(x_T), \label{eq:FHOCP} \tag{FH-OCP}\\
    \text{subject to:} &~ x_{k+1} = f(x_k, u_k), ~\text{and} ~ x_0 = x, \nonumber
\end{align}
where $\phi(\cdot)$ is a terminal cost function that is continuous and is such that $\phi(x) > 0, ~\forall ~x\neq 0$, and $\phi(x) = 0$ when $x=0$. 
We shall make the following assumptions for the rest of this section. 
\begin{assumption}{(A1)}\label{assump.1 cost}  
We assume that the cost function $c(x,u)$ has a global minimum at $(x,u) = (0,0)$, i.e., $\frac{\partial c}{\partial x} \Bigr|_{x=0, u=0} = 0$ and $\frac{\partial c}{\partial u}\Bigr|_{x=0, u=0} = 0$, $c(0,0) = 0$, and $c(x,u) > 0$ $\forall ~(x,u)\neq(0,0)$.
\end{assumption}
\begin{assumption}{(A2)}\label{assump.2 controllability} 
We assume that given any $x\in \mathcal{X}$, and any $\Omega \subset \mathcal{X}$, such that the origin is in $\Omega$, $\exists$ a control sequence $\{\bar{u}_k\}_{k=0}^{T(x)-1}$,  that ensures $\bar{x}_{T(x)} \in \Omega$ for some $T(x) < \infty$, under the dynamics defined above \eqref{eq:dynamics}. 
\end{assumption}
Assumption~\ref{assump.2 controllability} is a controllability assumption that ensures that any state can be controlled into entering the region $\Omega$ in finite time. 
We use the following definition for the set $\Omega$ in the rest of the paper: $\Omega_M = \{x ~|~ \phi(x) \leq M \}$, where $M$ is a parameter heuristically chosen depending on the application.

\begin{assumption}{(A3)}\label{assump.3 forward invariance}
    There exists a control policy $\pi(\cdot): \mathcal{X} \rightarrow \mathcal{U}$ that makes the set $\Omega_M$ forward invariant under the dynamics in \eqref{eq:dynamics}, i.e.,
    $f(x,\pi(x)) \in \Omega_M, ~\forall~x\in\Omega_M$. Also, let $c(x,\pi(x)) \eqcolon c^{\pi}(x) \leq \delta ~\forall~ x\in\Omega_M.$ Further $c(x,u) > \delta, ~\forall~ x \notin \Omega_M.$ Here, $\delta$ is a function of $M$, i.e., $\delta = \delta(M)$.
\end{assumption}

\begin{remark}
    If the system in \eqref{eq:dynamics} is linearly controllable around the origin, then the control policy in the set $\Omega$ can be taken as the linear quadratic regulator (LQR) policy, i.e. $\pi(x) = K_{lqr} x$ and the terminal cost can be replaced with the LQR cost-to-go, i.e., $\phi(x) = \tr{x} P x$, where $P$ is the calculated by solving the algebraic Riccati equation. This case is dealt with in detail in previous work \cite{mohamed2023infinitehorizon}. In this paper, we consider systems that are not linearly controllable around the origin. 
\end{remark}

\subsection{Existence of a Finite Horizon Solution}

We show below that the solution to \eqref{eq:FHOCP} cannot stay outside $\Omega_M$ for infinite time, and there exists a finite time at which the system will enter $\Omega_M$.
\begin{lemma}\label{lemma:finite time}
There exists a finite time $T(\Omega_M) < \infty$, such that the solution to \eqref{eq:FHOCP} with $T=T(\Omega_M)$ denoted by $(\bar{x}_k, \bar{u}_k)$ is such that $\phi(\bar{x}_{T(\Omega_M)}) \leq M$ for the first time, i.e., for $T< T(\Omega_M)$, $\phi(\bar{x}_k) > M$.
\end{lemma}
\begin{proof}
    We do a proof by contradiction.
    Let $\{x'_k\}_{k=0}^{T}$ be the solution to the \eqref{eq:FHOCP}. Consider the case where the terminal state $x'_T$ never enters the set $\Omega_M$ for any $T$. Since the cost $c(x,u)\geq\delta>0$ for $x\notin \Omega_M$, the cost $J^T_\infty \rightarrow \infty$ as $T\rightarrow \infty.$

    However, owing to A\ref{assump.2 controllability}, there exists a control sequence $\{\bar{u}_k\}_{k=0}^{T(x)-1}$ such that $\bar{x}_{T(x)}\in \partial\Omega_M$ (boundary of $\Omega_M$) for some finite $T(x)$. Let the cost of this trajectory be denoted as 
    \begin{align}\label{eq:cost_A2}
        \bar{J}(x) = \sum_{k=0}^{T(x)-1} c(\bar{x}_k, \bar{u}_k) + M,
    \end{align}
    where, we have substituted $\phi(\bar{x}_{T(x)}) = M$. Choose a $T$ such that $T > T(x)$, and $J^T_{\infty}(x) > \bar{J}(x) + \epsilon$, where $\epsilon$ is any positive number. For the sake of the proof we choose $\epsilon = (T-T(x))\delta$, where $\delta$ is as defined in A\ref{assump.3 forward invariance}, and the reason for this specific choice will be clear below. There is always a $T$ that will satisfy the above requirement since $J^T_{\infty}(x) \rightarrow \infty$ as $T\rightarrow \infty.$ 
    
    Now, apply the policy $\{\bar{u}_k \}_{k=0}^{T-1}$ to the system with $\bar{u}_k = \pi(\bar{x}_k) ~\forall ~k \geq T(x)$ (recall that $\pi(\cdot)$ is a policy that makes $\Omega_M$ invariant and we simply need its existence to prove the result, not know it per se). Let the cost of this trajectory be denoted as $\tilde{J}^T_{\infty}(x)$ and is given by $\tilde{J}^T_{\infty}(x) = \sum_{k=0}^{T(x) - 1} c(\bar{x}_k, \bar{u}_k) + \sum_{k=T(x)}^{T - 1} c(\bar{x}_k, \bar{u}_k) + \phi(\bar{x}_k)$.
    
    Using A\ref{assump.3 forward invariance} and using the fact that $\phi(\bar{x}_T)<M$ since $\bar{x}_T \in \Omega_M$, we can write 
    \begin{align}
        \tilde{J}^T_{\infty}(x) \leq \sum_{k=0}^{T(x) - 1} c(\bar{x}_k, \bar{u}_k) + \underbrace{(T - T(x))\delta}_{=\epsilon} + M.
    \end{align}
    Substituting \eqref{eq:cost_A2} in the above inequality, we get, $\tilde{J}^T_{\infty}(x) \leq \bar{J} + \epsilon$. 

    We know, $J^T_\infty(x)$ is the optimal cost for the \eqref{eq:FHOCP}, so $J^T_\infty(x) \leq \tilde{J}^T_{\infty}(x)$, and hence $J^T_\infty(x) \leq \bar{J} + \epsilon$. This contradicts the fact that we chose $T$ such that $J^T_\infty(x) > \bar{J} + \epsilon$. Thus, the solution to \eqref{eq:FHOCP} cannot stay outside $\Omega_M$ for all $T$, and there exists a finite $T$ such that the solution hits $\Omega_M$.
\end{proof}
\subsection{An Alternative Level Set based Construction}
We now define an alternate finite horizon construction to IH-OCP that will use the first hitting time to $\Omega_M$ as the time horizon and whose cost will satisfy the Bellman equation. The construction is suboptimal to the IH-OCP, but we show that the cost of this new construction converges to the true IH-OCP cost in the limit $M \rightarrow 0$.
We call this the alternate construction optimal control problem (AC-OCP), and it is defined as:
\begin{align}
    J^M_\infty(x) &= \min_{ \{u_k\}_{k=0}^{T-1} , T} \sum_{k=0}^{T-1} c(x_k, u_k) + \max(\phi(x_T), M) \label{eq:ACOCP}\tag{AC-OCP} \\
    \text{subject to:} &~x_{k+1} = f(x_k, u_k),\nonumber \\
    & x_T \in \Omega_M,  ~\text{and given}~ x_0 = x.  \nonumber
\end{align}

\textit{Note: The above problem has a free final time $T$ that needs to be optimized over in conjunction
with the control actions. The free final time will prove crucial to showing the cost function is a CLF and it converges to the optimal IH-OCP cost.}

We now prove the following result.
\begin{lemma}\label{lemma:first hitting time}
    The optimal time $T^*$ to the \eqref{eq:ACOCP} is the first hitting time of the set $\Omega_M$ for the \eqref{eq:FHOCP}, i.e., $T^* = T(\Omega_M).$
\end{lemma}
\begin{proof}
    Let the solution to \eqref{eq:ACOCP} be denoted by $(\tilde{x}_k, \tilde{u}_k)$ and the cost be given by:
    \begin{align}\label{eq:JM_optimal}
        J^{M}_\infty(x) = \sum_{k=0}^{T^* - 1} c(\tilde{x}_k, \tilde{u}_k) + M. 
    \end{align}
    We consider two cases: $T^* > T(\Omega_M)$ and $T^* < T(\Omega_M).$\\
    1) $T^* > T(\Omega_M), ~\tilde{x}_{T^*} \in \Omega_M$: \\
    We can say the following from \eqref{eq:JM_optimal}:
    \begin{align}\label{eq:JM_krunc}
        J^M_\infty(x) > \sum_{k=0}^{T(\Omega_M) - 1} c(\tilde{x}_k, \tilde{u}_k) + M, 
    \end{align}
    because $J^M_\infty(x)$ will have the additional $\sum_{k=T(\Omega_M)}^{T^* - 1} c(\tilde{x}_k, \tilde{u}_k)$ terms. 
    We also know that optimal cost of \eqref{eq:FHOCP} with horizon $T(\Omega_M)$ will satisfy
    \begin{align}\label{eq:JT inequality}
        J^{T(\Omega_M)}_\infty \leq \sum_{k=0}^{T(\Omega_M) - 1} c(\tilde{x}_k, \tilde{u}_k) + M,
    \end{align}
    because $J^{T(\Omega_M)}_\infty$ is the optimum for the \eqref{eq:FHOCP}. From \eqref{eq:JM_krunc} and \eqref{eq:JT inequality}, we can say $J^M_\infty(x) > J^{T(\Omega_M)}_\infty$, which contradicts the fact that $J^M_\infty(x)$ is the optimum for \eqref{eq:ACOCP}, and hence $T^*$ cannot be the optimal time.\\
    2) $T^* < T(\Omega_M):$ \\
    We know $T(\Omega_M)$ is the first hitting time. Hence for any $T < T(\Omega_M)$, $\tilde{x}_{T} \notin \Omega_M$, which is a constraint the solution to \eqref{eq:ACOCP} has to satisfy. Hence, the optimal time $T^*$ cannot be less than $ T(\Omega_M).$
\end{proof}

The intuition behind the above lemma is that the system incurs more cost when its solution is in the interior of set $\Omega_M$ due to the $\max$ function, i.e., $\max(\phi(x_T), M) = M$ as $\phi(x_T) < M$. Hence the optimal solution will be the one that stops at the boundary of set $\Omega_M$ when $\phi(x_T)=M$, which is the solution with time horizon as the first hitting time $T(\Omega_M).$

Now, we go on to show that the cost of AC-OCP satisfies the Bellman equation.
\begin{lemma}\label{lemma:Bellman equation}
    The cost-to-go of AC-OCP $J^M_\infty(x)$ satisfies the Bellman equation for all initial states $x \notin \Omega_M$.
\end{lemma}
\begin{proof}
    The optimal cost of AC-OCP can be written as 
    \begin{align*}
        J^M_\infty(x) &= \min_{ \{u_k\}_{k=0}^{T-1}, T } \Big[ c(x, u_0) + \sum_{k=1}^{T-1} c(x_k, u_k) \\
        &+ \max(\phi(x_T), M) \Big].
    \end{align*}
    The above equation can also be written as:
    \begin{align*}
        J^M_\infty(x) &=  \min_{u_0}\Big[ c(x, u_0) + \min_{ \{u_k\}_{k=1}^{T-1}, T }[\sum_{k=1}^{T-1} c(x_k, u_k) \\
        &+ \max(\phi(x_T), M)] \Big],\\
        J^M_\infty(x) &=\min_{u_0}\Big[ c(x, u_0) + J^M_\infty(f(x,u_0)) \Big].
    \end{align*}
    The above can be shown for any initial state $x \notin \Omega_M$.
\end{proof}
\begin{corollary}
    If there exists some $M>0$ such that the set $\Omega_M$ is forward invariant for the uncontrolled dynamics or if we have knowledge of a policy $\pi(\cdot)$ that renders the closed loop invariant with respect to $\Omega_M$, $J_{\infty}^M(\cdot)$ is a CLF for the given system \eqref{eq:dynamics}, and its policy renders the set $\Omega_M$ asymptotically stable. 
\end{corollary}

In the following theorem, we show that the cost of AC-OCP converges to the optimal IH-OCP cost as $M\rightarrow 0.$
\begin{theorem}\label{theorem:cost convergence}
    The AC-OCP cost $ J^M_\infty(x)$ converges to the IH-OCP cost $J^*_{\infty}(x)$ in the limit $M \rightarrow 0$, i.e., $$ \lim_{M\rightarrow 0 } J^M_\infty(x) = J^*_{\infty}(x),$$ assuming that $J^*_\infty(\cdot)$ is continuous at the origin.
\end{theorem}
\begin{proof}
    Let $(\tilde{x}_k, \tilde{u}_k)$ denote the solution to \eqref{eq:ACOCP}, and $(x^*_k, u^*_k)$ denote the solution to \eqref{eq:IHOCP}. Now, we compare the costs by applying the AC-OCP policy $\{\tilde{u}_k\}_{k=0}^{T(\Omega_M)}$ to the IH-OCP. Since the AC-OCP policy is only defined for $T(\Omega_M)$ steps, we assume $\tilde{u}_k = u^*_k$ for $k\geq T(\Omega_M)$ for the sake of argument, and the result still holds for any policy that renders the set $\Omega_M$ forward invariant for the system dynamics. Since $J^*_\infty(x)$ is the optimal for \eqref{eq:IHOCP}, the cost of any other policy satisfies,
    \begin{align}\label{eq:lemma4_1}
        J^*_\infty(x) \leq \sum_{k=0}^{T(\Omega_M) - 1} c(\tilde{x}_k, \tilde{u}_k) + \sum_{k=T(\Omega_M)}^{\infty} c(\tilde{x}_k, u^*_k).
    \end{align}
    Using the knowledge that $J^M_\infty(x) = \sum_{k=0}^{T(\Omega_M) - 1} c(\tilde{x}_k, \tilde{u}_k) + M$, and  $\sum_{k=T(\Omega_M)}^{\infty} c(\tilde{x}_k, u^*_k) = J^*_\infty(\tilde{x}_{T(\Omega_M)})$, we can write \eqref{eq:lemma4_1} as,
    \begin{align*}
        J^*_\infty(x) \leq J^M_\infty(x) - M + J^*_\infty(\tilde{x}_{T(\Omega_M)}).
    \end{align*}
    Restructuring the equation and taking the limit $M \rightarrow 0$ gives, 
    \begin{align*}
        \lim_{M\rightarrow 0 } (J^*_\infty(x) - J^M_\infty(x)) \leq \lim_{M\rightarrow 0 } (J^*_\infty(\tilde{x}_{T(\Omega_M)}) -M),
    \end{align*}
    where, $\lim_{M\rightarrow 0 } J^*_\infty(\tilde{x}_{T(\Omega_M)}) = J^*_\infty( \lim_{M\rightarrow 0 } \tilde{x}_{T(\Omega_M)})$, as $J^*_\infty(\cdot)$ is continuous at the origin. As $M\rightarrow 0, ~\Omega_M$ will shrink in size and in the limit, $\Omega_M= \{0\}$ (since only $x = 0$ satisfies the condition $\phi(x)\leq 0,$ and please note the distinction between the number 0 and the state space origin 0.) Hence, in the limit $\tilde{x}_{T(\Omega_M)} = 0$ due to the terminal state constraint $\tilde{x}_{T(\Omega_M)} \in \Omega_M$ in \eqref{eq:ACOCP}, which implies $\lim_{M\rightarrow 0 } J^*_\infty(\tilde{x}_{T(\Omega_M)}) = J^*_\infty(0) = 0$. Thus, 
    \begin{align}\label{eq:lemma4_ineq1}
        \lim_{M\rightarrow 0 } J^M_\infty(x) \geq J^*_\infty(x).
    \end{align}

    Similarly, we can show $\lim_{M\rightarrow 0 } (J^M_\infty(x) - J^*_\infty(x)) \leq 0$ by applying the IH-OCP policy $\{ u^*_k \}_{k=0}^{\infty}$ to \eqref{eq:ACOCP}. Due to the optimality of  $J^M_\infty(x)$, we get $J^M_\infty(x) \leq \sum_{k=0}^{T(\Omega_M) - 1} c(x^*_k, u^*_k) + \max(\phi(x^*_{T(\Omega_M)}, M)= J^*_\infty(x) - J^*_\infty(x^*_{T(\Omega_M)}) + \max(\phi(x^*_{T(\Omega_M)}, M)$.
    Rearranging and taking the limit gives,
    \begin{align*}
       \lim_{M\rightarrow 0 } \Big(J^M_\infty(x) -  J^*_\infty(x)\Big) &\leq \lim_{M\rightarrow 0 } \Big(- J^*_\infty(x^*_{T(\Omega_M)})\\
       &+ \max(\phi(x^*_{T(\Omega_M)}, M)\Big) 
    \end{align*}
    As shown previously $\lim_{M\rightarrow 0 } J^*_\infty(x^*_{T(\Omega_M)}) = 0$. If $\max(\phi(x^*_{T(\Omega_M)}, M) = \phi(x^*_{T(\Omega_M)})$, then $\lim_{M\rightarrow 0 } \phi(x^*_{T(\Omega_M)}) =  \phi(\lim_{M\rightarrow 0 } x^*_{T(\Omega_M)})$, since $\phi(\cdot)$ is also a continuous function. Using the similar argument used for $J^*_\infty(\cdot),$ we can say $\phi(\lim_{M\rightarrow 0 } x^*_{T(\Omega_M)})=0$. If $\max(\phi(x^*_{T(\Omega_M)}, M) = M$, it is trivial to show the limit is $0$. Hence, we get
    \begin{align} \label{eq:lemma4_ineq2}
       \lim_{M\rightarrow 0 } J^M_\infty(x)  &\leq   J^*_\infty(x) .
    \end{align}
    From \eqref{eq:lemma4_ineq1} and \eqref{eq:lemma4_ineq2}, we get $\lim_{M\rightarrow 0 } J^M_\infty(x) = J^*_\infty(x).$ 
\end{proof}

The intuition for the above proof is that as $M \rightarrow 0$, the set $\Omega_M$ shrinks in size and the state at the first hitting time $x_{T(\Omega_M)} \rightarrow 0$, in which case, the AC-OCP and IH-OCP become equivalent problems. 

\subsection{Discussion} \label{sec:discussion}
\paragraph{Why propose the AC-OCP?}

The purpose of proposing the AC-OCP is that it captures the essence of the IH-OCP in that the problem determines the transfer time, for which it has to be free, as opposed to fixed, as in FH-OCP. Moreover, the transfer time is not unique; it varies with the initial state in that different initial states would need different transfer times for optimal performance. Also, the AC-OCP construction helps us guarantee that the finite optimal time cost-to-go is a CLF that renders the set $\Omega_M$ globally asymptotically stable. 

\paragraph{How would one solve the AC-OCP?}  

We do not solve AC-OCP directly. The solution is given by solving \eqref{eq:FHOCP} by sweeping for different values of the time horizon $T$ until we find the time $T^*$, for which the solution enters the set $\Omega_M$, i.e., the terminal cost of the solution satisfies $\phi(x_T) \leq M$. This sweep of different values of $T$ can be done in parallel, and the search can be optimized to find the $T^*.$

\paragraph{How is this different from \cite{mohamed2023infinitehorizon}?}
In \cite{mohamed2023infinitehorizon}, the stationary optimal cost function, obtained by solving the stationary Riccati equation, was the obvious candidate for the terminal cost since it is an arbitrarily good approximation of the true optimal cost as the terminal set gets small. In lieu, in this work, because of the absence of linear controllability, we use the heuristic terminal cost $\phi(\cdot)$ which only has the property that $\phi(x) \rightarrow 0$ as $x \rightarrow 0$. The advantage in the linearly controllable case is that the terminal set for which the linear controller is a CLF/ good approximation can be quite large, thereby leading to a significant computational saving in solving the problem when compared to solving it without the terminal cost. As we shall show in our computational experiments, the heuristic terminal cost regularization is necessary to solve complex nonholonomic problems such as the fish and swimmer models.

\paragraph{Contrast with Nonlinear MPC}
Traditional nonlinear MPC has a fixed horizon $N$, and it replans over the same fixed horizon at every step to furnish a time-invariant control law \cite{mayne2000constrained}. This has the implication that the MPC policy only renders states that can be controlled to the terminal set in at most $N$ steps asymptotically stable, leading to a small region of attraction. In contrast, we solve the problem from any initial state over a free horizon and this is precisely what allows for the GAS nature of the resulting policy. Further, this obviates the need for replanning in our approach. Also, for a given initial state, the horizon $N$ that must be used to obtain the MPC control law is not clear. Using a heuristic horizon will lead to suboptimal results. Moreover, the horizon $N$ differs based on the initial state. These aspects will be shown in the numerical results. 

A drawback of solving AC-OCP  by sweeping through the time horizon to identify the transfer $T$ is that it is computationally more expensive than just solving MPC at one-time step. However, MPC requires replanning to ensure the system enters the terminal set which would also demand computational resources at every time step. While, the solution obtained from AC-OCP ensures the system enters the terminal set at the transfer time and does not require replanning. With current advances in computing resources, the sweep through time horizons is feasible to be done in real-time.   

\section{Solution to the Discounted Infinite Horizon Optimal Control Problem}\label{section:sol_DIHOCP}

Reinforcement learning problems for continuous control predominantly consider discounted infinite horizon problems \cite{lillicrap2015continuous}. In this section, we explore the discounted problem and show that we can use a finite horizon construction similar to the previous section. The discounted problem \cite{bertsekas_vol2} is defined as 
\begin{align}
    J^*_\infty(x) &=\min_{u_k}\sum_{k=0}^{\infty} \beta^k c(x_k,u_k),
\label{eq:D-IHOCP}\tag{D-IHOCP} \\
&\text{subject to:}  ~x_{k+1} = f(x_k, u_k), ~\text{and}~x_0 = x, \nonumber
\end{align}
where discount factor $\beta \in(0,1)$.

Similar to Section~\ref{section:sol_IHOCP}, we use an alternate construction with a free final time for the discounted finite horizon optimal control problem:
\begin{align}
    J_\infty^M(x) &=\min_{u_k, T}\sum_{k=0}^{T-1} \beta^k c(x_k,u_k)+\beta^T \max(\phi(x_k), M),
\label{eq:D-ACOCP}\tag{D-ACOCP} \\
&\text{subject to:}  ~x_{k+1} = f(x_k, u_k). \nonumber
\end{align}
where $\phi(\cdot)$ is some terminal cost function, $\Omega_M=\{x:\phi(x)\leq M\}$ and $M<\infty$ is some number.

We invoke assumptions A\ref{assump.1 cost}, A\ref{assump.2 controllability}, and A\ref{assump.3 forward invariance} as established in Section~\ref{section:sol_IHOCP} for the results below. We will prove results for the discounted problem analogous to the undiscounted case. First, we will show that given any initial set $\Omega^0$, there always exists a $\beta<1$ such that any $x_0\in \Omega^0$ may be controlled into the terminal set $\Omega_M$ (see Fig.~\ref{fig.discounted_illustration}).

\begin{figure}[!htbp]
    \centering
    \def\svgwidth{0.6\columnwidth}
    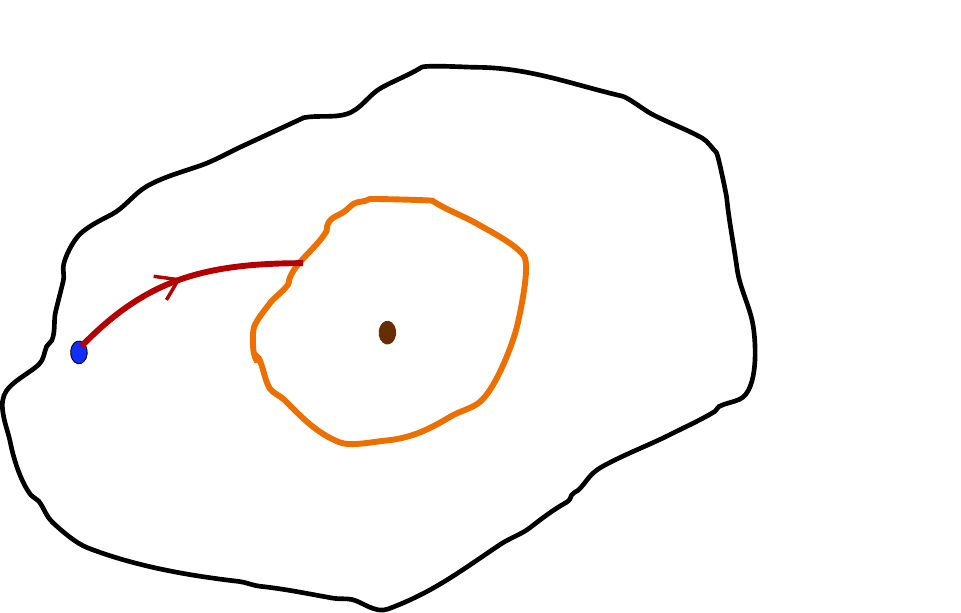
    \caption{An illustration of the discounted cost problem. We will show that given any $\Omega^0$, there exists a $\beta<1$, s.t., any $x_0 \in \Omega^0$ may be controlled into the terminal set $\Omega_M$.} 
    \label{fig.discounted_illustration}
\end{figure}
\begin{lemma}
    Given any $x_0\in \Omega^0$, there exists a discount factor $\beta(x_0)<1$ such that, given sufficient time, the solution to \eqref{eq:D-ACOCP} will enter the set $\Omega_M$.
\end{lemma}

\begin{proof}
    We drop the explicit dependence on $x_0$ in the following for convenience. We do the proof by contradiction. Recall from A\ref{assump.2 controllability}; there exists a control sequence $\{\bar{u}_k \}_{k=0}^{T(x)-1}$ which enters $\Omega_M$ at some finite time $T(x)$. Let $T(x) = \bar{T}$ for the given initial state $x$. The trajectory is denoted by $(\bar{x}_k,\bar{u}_k)$, and let $\bar{J}$ be the cost associated with it, assuming that the policy $\pi(\cdot)$ is applied once the trajectory enters $\Omega_M$. Note that since $c^\pi(x)<\delta$, the tail cost is finite, i.e.,  $\sum_{k>\bar{T}}^{\infty} c^\pi(\Bar{x}_k)\beta^k<\infty$.

    Suppose that the solution to \eqref{eq:D-ACOCP} never enters $\Omega_M$. Then, the cost of the policy from $\bar{T}$ till some $T > \bar{T}$ is $\sum_{k=\bar{T}}^T \beta^k c(x_k,u_k) = \beta^{\bar{T}}\sum_{k=0}^{T-\bar{T}} \beta^k c(x_{k+\bar{T}},u_{k+\bar{T}}) \geq \beta^{\bar{T}}\sum_{k=0}^{T-\bar{T}} \beta^k \delta = \beta^{\bar{T}} \Bigg( \frac{1-\beta^{(T-\bar{T})}}{1-\beta} \Bigg) \delta.$

We now show that it is always possible to find a $\beta, T$ such that $\beta^{\bar{T}} \Big( \frac{1-\beta^{(T-\bar{T})}}{1-\beta} \Big) \delta>\bar{J}+\beta^T M$. 

The above implies that: $\Bigg( \frac{1-\beta^{(T-\bar{T})}}{1-\beta} \Bigg)>\frac{\bar{J}+\beta^T M}{\delta \beta^{\bar{T}}}$,
for some $(\beta, T)$. Now, consider the function $f(\beta,T)= \frac{1-\beta^{(T-\bar{T})}}{1-\beta}$.
This function is continuous in $(\beta,T)$, and $\lim_{T\rightarrow\infty,\beta\rightarrow 1}f(\beta,T)\rightarrow\infty$. By definition, this implies that $\exists~(\beta, T)$ s.t. $f(\beta,T)>\frac{\bar{J}+\beta^T M}{\delta \beta^{\bar{T}}}$. However, this implies that the \eqref{eq:D-ACOCP} optimal cost corresponding to the time $T$, say $J_\infty^{M,T}>\bar{J}+\beta^T M$, thereby contradicting the fact that $J_\infty^{M,T}$ is optimum. Note that $(\bar{J}+\beta^T M)$ is an upper bound on the cost of the nominal policy $\bar{u}_k$ with the terminal policy $\pi(\cdot)$. Thereby, this implies that the solution to the \eqref{eq:D-ACOCP} has to enter $\Omega_M$ for some finite time, given $\beta$ is sufficiently close to $1$.
\end{proof}

We assume the following to remove the dependence on $x_0$ for $\beta$.
\begin{assumption}\label{assump.d1}
    Let $\sup\limits_{x_0\in \Omega^0} \bar{T}(x_0)<\bar{T}<\infty$, and let $\sup\limits_{x_0\in \Omega^0} \bar{J}(x_0)<\bar{J}<\infty$. 
\end{assumption}
Then, if we choose $\beta, T$ s.t. $\Big( \frac{1-\beta^{(T-\bar{T})}}{1-\beta} \Big)>\frac{\bar{J}+\beta^T M}{\delta \beta^{\bar{T}}}$, then the solution to \eqref{eq:D-ACOCP} hits the set $\Omega_M$ in finite time for any $x_0\in \Omega^0$.
\begin{corollary}
    Under Assumption \ref{assump.d1}, there exists a finite $\beta<1$ s.t. the solution to the \eqref{eq:D-ACOCP} enters $\Omega_M$ in finite time.
\end{corollary}

Now, we will show that the solution to  \eqref{eq:D-ACOCP} gives the first hitting time of the set $\Omega_M$.
\begin{lemma}
    The optimal time $T^*$ for the \eqref{eq:D-ACOCP} is the first hitting time of the set $\Omega_M$, i.e., $T^*=T(\Omega_M)$.
\end{lemma}
\begin{proof}
    The proof is identical to the proof of Lemma~\ref{lemma:first hitting time}.
\end{proof}  

Now, we show that the cost-to-go of the alternate construction will satisfy the discounted Bellman equation.
\begin{lemma}
    The cost-to-go of \eqref{eq:D-ACOCP} $J_\infty^M(x)$ satisfies the discounted Bellman equation for all initial states $x\notin \Omega_M$.
\end{lemma}

\begin{proof}
    The optimal cost of \eqref{eq:D-ACOCP} can be written as 
    \begin{align*}
        J^M_\infty(x) &= \min_{ \{u_k\}_{k=0}^{T-1}, T } \Big[ \beta^0 c(x, u_0) + \sum_{k=1}^{T-1} \beta^k c(x_k, u_k) \\
        &+ \beta^T\max(\phi(x_T), M) \Big].
    \end{align*}
    The above equation can also be written as:
    \begin{align*}
        J^M_\infty(x) &=  \min_{u_0}\Big[ c(x, u_0) +  \min_{ \{u_k\}_{k=1}^{T-1}, T }\beta[\sum_{k=1}^{T-1} \beta^{k-1}c(x_k, u_k) \\
        &+ \beta^{T-1}\max(\phi(x_T), M)] \Big],\\
        J^M_\infty(x) &=\min_{u_0}\Big[ c(x, u_0) + \beta J^M_\infty(f(x,u_0)) \Big].
    \end{align*}
    The above can be shown for any initial state $x \notin \Omega_M$ and time step $k$. 
\end{proof}
\begin{remark}
    Note that $\beta<1$ only for some $\Omega^0 \subset \mathcal{X}$, and thus, the discounted policy cannot be globally asymptotically stable.
\end{remark}

Finally, we will show that the cost-to-go of the alternate discounted problem converges to the optimum cost of D-IHOCP.
\begin{theorem}
    The D-ACOCP cost $ J^M_\infty(x)$ converges to the discounted infinite-horizon OCP cost $J^*_{\infty}(x)$ in the limit $M \rightarrow 0$, i.e., $$ \lim_{M\rightarrow 0 } J^M_\infty(x) \rightarrow J^*_{\infty}(x).$$
\end{theorem}
\begin{proof}
    The proof is essentially identical to the undiscounted case discussed in Theorem~\ref{theorem:cost convergence}. 
\end{proof}

\begin{remark}
    Note that finding the right $\beta$ given some set $\Omega^0$ is practically infeasible, as the requisite $\bar{T}$, $\bar{J}$ etc., are unknown in general. Thus, the above result is strictly an existence result, and has no practical way of implementation. In practice, the process is reversed: we choose a $\beta$ and such a choice may have a small region of attraction $\Omega^0$ resulting in policies that are myopic.
\end{remark}
\section{Empirical Results}\label{section:results}
In this section, we present the empirical results. The proposed theory is extended to a Car-like robot (4 states, 2 inputs) and the MuJoCo-based simulator for the Fish robot (27 states, 6 inputs). We show that the cost converges as the horizon or the transfer time is increased. We also show the dependence of the transfer time on the initial conditions. In Section~\ref{sec:mpc}, we compare our approach with MPC. We only show experiments for the undiscounted case here due to the paucity of space. 

\subsection{System Description}
The Car-like robot has well-established nonlinear dynamics and is simulated in MATLAB. The dynamics is given by $\dot{p_x} = v~cos\theta$ , $\dot{p_y} = v ~sin\theta $, $\dot{\theta} = v ~tan(\delta)/L$, $\dot{v} = a$, where the control inputs are the acceleration $a$ and steering angle $\delta$ and the state is $x = [p_x, p_y, \theta, v]$.
Given an initial condition $x_0$, the task is to drive the system to the desired terminal state $x_T$. Similarly, the Fish-robot also starts at the origin, with the target coordinates for its head specified. The initial and final states of the fish model can be observed in Fig.~\ref{figinit}. All these systems are nonholonomic, and thus, the terminal cost cannot be the cost-to-go of the linearized system as proposed in our previous work \cite{mohamed2023infinitehorizon}, and we chose a heuristic cost in this case. The heuristic terminal cost used in the experiments is a quadratic cost on the state error, i.e. $\phi(x) = \frac{1}{2} x^T S x$, where $S$ is a diagonal matrix with non-negative entries. 

\begin{figure}[!htbp]
\centering
      \subfloat{\includegraphics[width=0.3\linewidth]{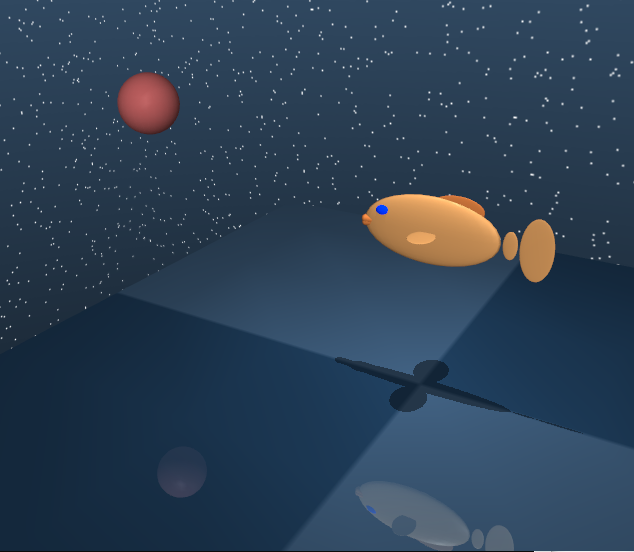}}
       \subfloat{\includegraphics[width=0.3\linewidth]{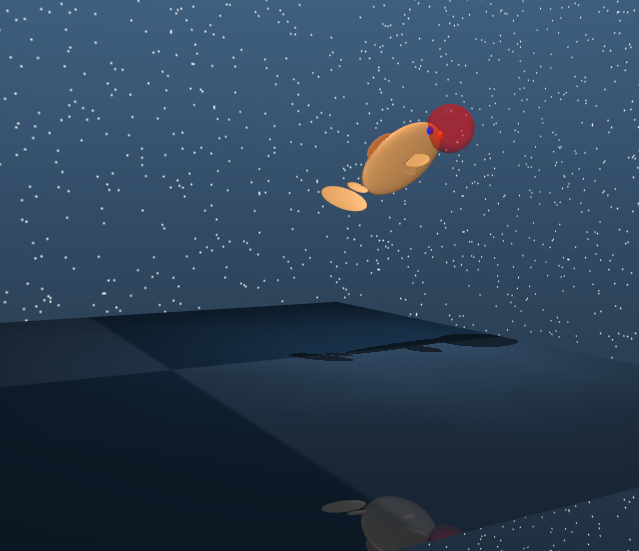}}
\caption{Fish Model simulated in MuJoCo in their initial and final states.}
\label{figinit}
\vspace{-0.5cm}
\end{figure}

\begin{figure}[!htbp]
\centering
    \sbox0{\subfloat[Car-like - Total cost - Case 1]{\includegraphics[width=0.48\linewidth]{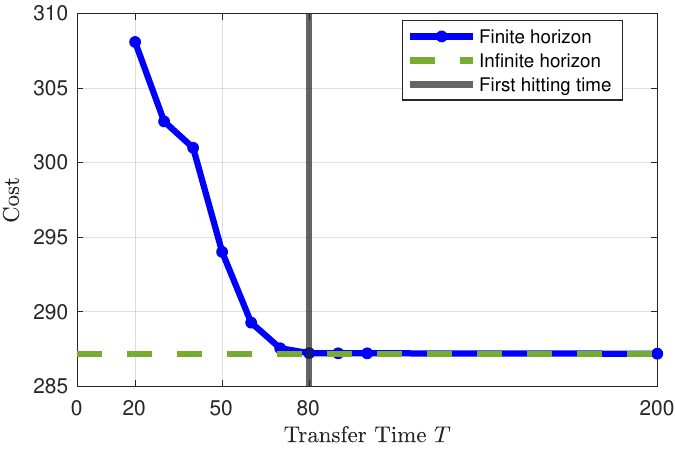}}}
    \sbox1{\subfloat[Car-like - Terminal cost - Case 1]{\includegraphics[width=0.48\linewidth]{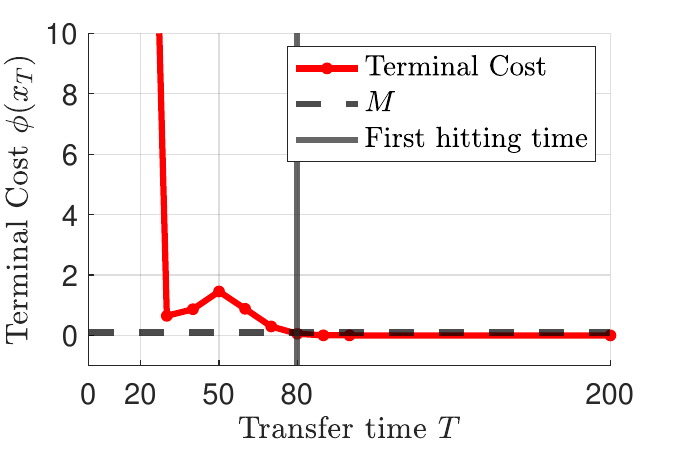}}}
    \sbox2{\subfloat[Car-like - Total cost - Case 2]{\includegraphics[width=0.48\linewidth]{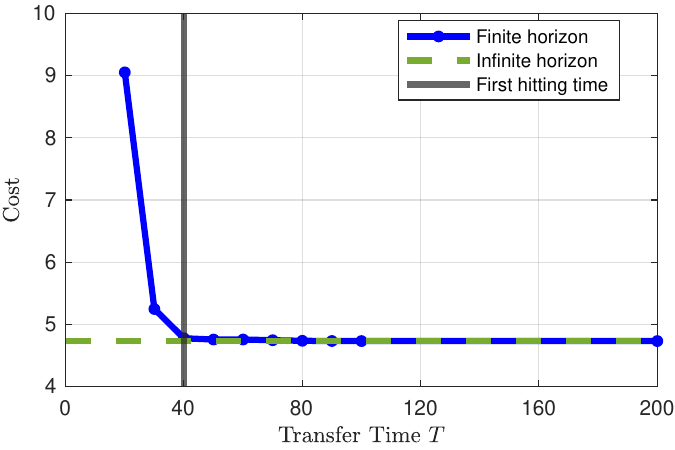}}}
    \sbox3{\subfloat[Car-like - Terminal cost - Case 2]{\includegraphics[width=0.48\linewidth]{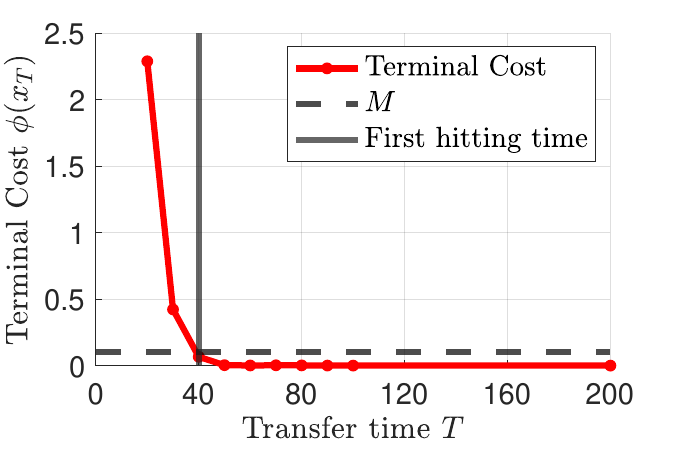}}}

    \sbox4{\subfloat[Fish - Total cost - Case 1]{\includegraphics[width=0.48\linewidth]{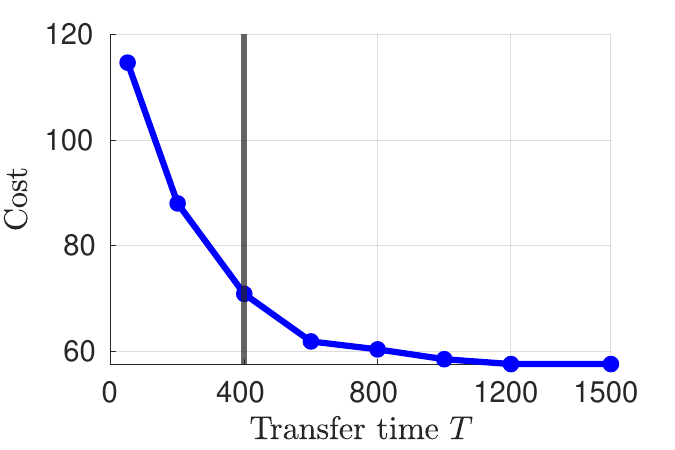}}}
    \sbox5{\subfloat[Fish - Terminal cost - Case 1]{\includegraphics[width=0.48\linewidth]{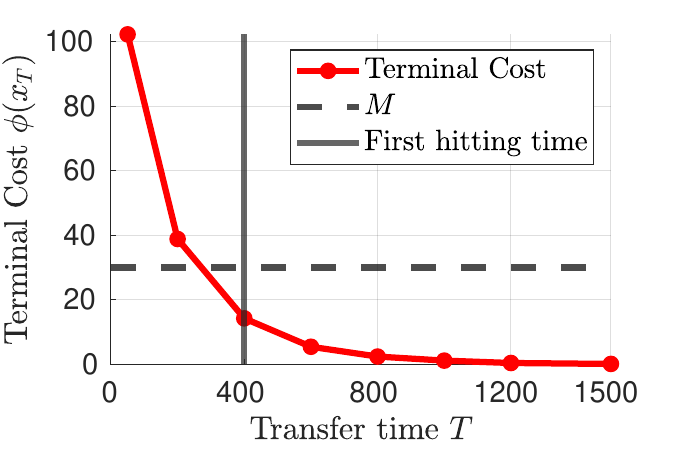}}}
    \sbox6{\subfloat[Fish - Total cost - Case 2]{\includegraphics[width=0.48\linewidth]{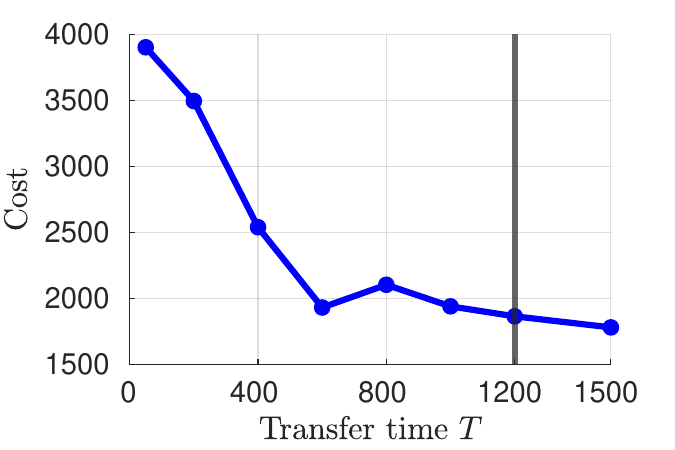}}}
    \sbox7{\subfloat[Fish - Terminal cost - Case 2]{\includegraphics[width=0.48\linewidth]{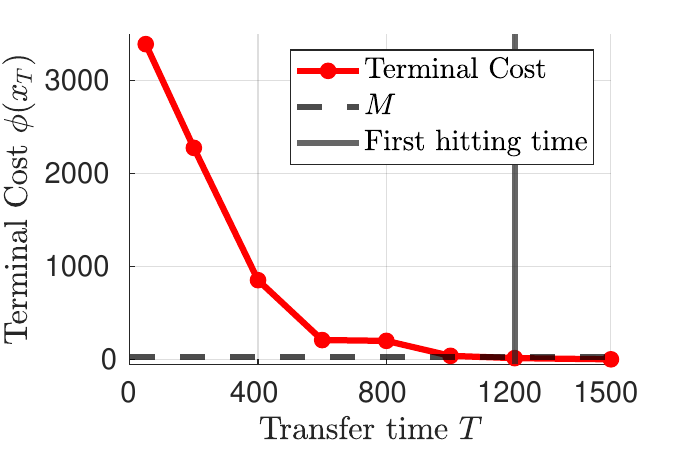}}}
    \centering
    \usebox0\hfil \usebox1\par
  \usebox2\hfil \usebox3\par
  \usebox4\hfil \usebox5\par
  \usebox6\hfil \usebox7\par
    \caption{Results for the car-like robot (a)-(d), and the fish model (e)-(h) with two different initial conditions - labeled Case 1 and 2. The parameters of the simulation are shown in Table~\ref{tab:experiments}. It is observed that different initial conditions correspond to different transfer times in both cases.}
    \label{fig:fish}
\end{figure}
\subsection{Optimality of AC-OCP and computing the transfer time}
To empirically verify our proposed approach, we need to show that the cost of the AC-OCP converges to the infinite horizon optimal cost. We use the iterative Linear Quadratic Regulator (iLQR) algorithm \cite{ILQG_tassa2012synthesis} to solve for the nonlinear optimal control and its corresponding optimal cost. For smooth nonlinear systems with control affine dynamics and a quadratic control cost, it can be shown that the iLQR algorithm will converge to the unique global optimum \cite{arxiv_D2C2.0} for a sufficiently small time discretization, thus circumventing the issue of multiple local minima. 

The iLQR incremental cost parameters and the terminal cost $\phi(x_T)$  are suitably chosen, and the optimization problem is set up with a finite horizon or `transfer time' $T$. Next, we sweep the transfer time $T$, until and beyond the first-hitting time $T(\Omega_M)$, where $x_{T(\Omega_M)} \in \Omega_M$. A small value is chosen for $M$, such that the terminal cost is arbitrarily close to zero, \textit{i.e.} the system is very close to the target state, and thus, the set $\Omega_M$ is forward invariant for all practical purposes. Since we do not have the solution for the IH-OCP, the infinite-horizon optimal cost is computed by taking a long enough horizon for each problem without the terminal cost and using iLQR to solve the optimization. The infinite horizon cost for the car-like robot was calculated with a horizon of $500$. The fish model fails to converge without a terminal cost, so we do not plot the infinite horizon cost for those systems. This is another reason to have a regularizing terminal cost for complex systems for stability, in addition to the free final time. It is observed that for any $T>T(\Omega_M)$, the cost $J^T_\infty(x)$ converges to the true optimal cost of the IH-OCP for the car-like robot, while still converging to the forward-invariant set for the more complex cases (Fig.~\ref{fig:fish}). It is also observed that, for smaller horizons less than the first hitting time, the terminal cost remains high, and correspondingly the system fails to converge to the target state. The experiments, thus, empirically validate Lemma~\ref{lemma:first hitting time}, wherein increasing the horizon past the first hitting time does not lead to a significant decrease in the cost, and hence, the first hitting time is sufficient to reach the goal set.

In order to observe the dependence of the transfer time on the initial state, the above experiment is repeated for different initial conditions for the Fish and Car-like robot systems. The initial and target states for the corresponding cases are tabulated in Table~\ref{tab:experiments}. We observe clearly that the hitting time $T(\Omega_M)$ depends on the initial condition of the system (Fig.~\ref{fig:fish}).

\begin{table}[!htbp]
    \centering
    \begin{tabular}{|c|c|c|c|c|}
        \hline
         System & Case \#& Initial state & Terminal state & $M$  \\
         \hline 
         \hline
         Car-like&1 & $(6,-6,\pi/3,10)$ & $(10,5,0,0)$ & $0.1$ \\
         \hline
         Car-like &2 & $(10.46,6.46,1.26,-0.33)$ & $(10,5,0,0)$ & $0.1$\\
         \hline
         Fish&1 & $(0.5,0.25,1)$  & $(0.4,0.2,1)$ & $30$\\
         \hline
         Fish&2 & $(-0.2,-0.1,-0.5)$ & $(0.4,0.2,1)$ & $30$\\
         \hline
    \end{tabular}
    \caption{Initial and target states for the Car-like robot and Fish-robot, for observing the dependence of transfer time on initial state.}
    \label{tab:experiments}
\end{table}

\subsection{Comparison with MPC}\label{sec:mpc}

Nonlinear MPC is widely used to solve infinite horizon problems \cite{mayne2014model}. As mentioned in the discussion in Section~\ref{sec:discussion}, the horizon used to solve the open-loop optimal control problem in MPC is unclear in the literature. In order to study the effect of horizon selection, we perform an experiment and test out different values of the horizon in the car-like system, evaluating the performance in terms of the cost and the time taken to transfer inside the terminal set. In Figure~\ref{fig:mpc}, we compare different MPC policies labeled as `\textit{MPC}-$N$', where $N$ stands for the horizon used to solve the open-loop problem for the corresponding MPC policy. We do this for two cases, corresponding to different initial states of the system mentioned in Table~\ref{tab:experiments}. The cost for the MPC policy is computed by running the MPC policy until time $T$, calculating the cost incurred and the terminal cost due to the error in the state at time $T$. As indicated in Figure~\ref{fig:mpc}-(a) and \ref{fig:mpc}-(c), the MPC policy with horizon 20 (MPC-20) converges to a suboptimal cost in both cases. MPC-40 also converges to a suboptimal cost for Case 1 while it converges to the optimum infinite horizon cost in Case 2. MPC-80 converges to the optimum cost in both cases. Hence, choosing a heuristic horizon leads to suboptimal performance. In Figure~\ref{fig:mpc}-(b) and \ref{fig:mpc}-(d), we show when the system enters the terminal set $(\phi(x)\leq M)$ under the different policies, and the corresponding transfer times are shown in Table~\ref{tab:transfer_times}. As the table shows, MPC policy using a horizon $N$ does not necessarily enter the terminal set in $N$ steps. Hence, the \ref{eq:ACOCP} construction, wherein we use a free final time, ensures optimality as well as stability guarantees to enter the terminal set. The free final time is crucial since different initial states will have different optimal transfer times, as corroborated by the two cases shown for the car-like system.

\begin{figure}[!htbp]
\centering
    \sbox0{\subfloat[Car-like - Total cost - Case 1]{\includegraphics[width=0.5\linewidth]{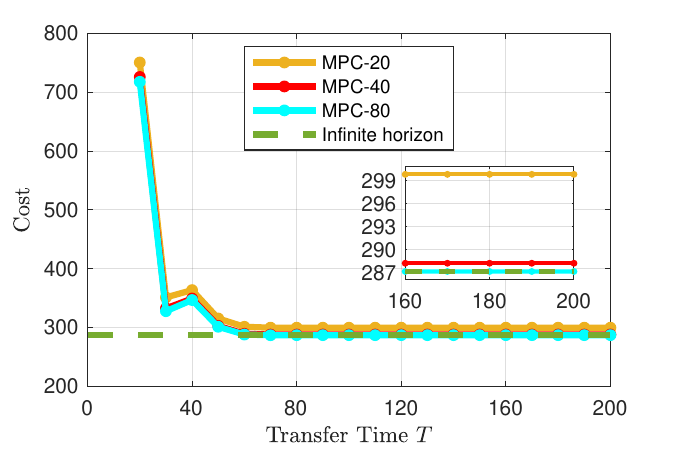}}}
    \sbox1{\subfloat[Car-like - Terminal cost - Case 1]{\includegraphics[width=0.5\linewidth]{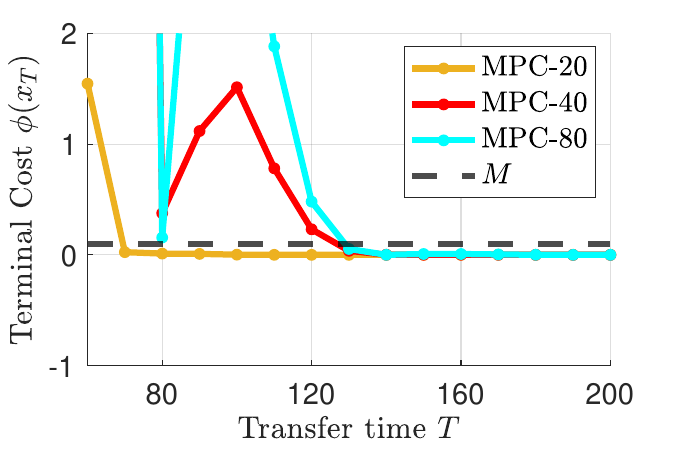}}}
    \sbox2{\subfloat[Car-like - Total cost - Case 2]{\includegraphics[width=0.5\linewidth]{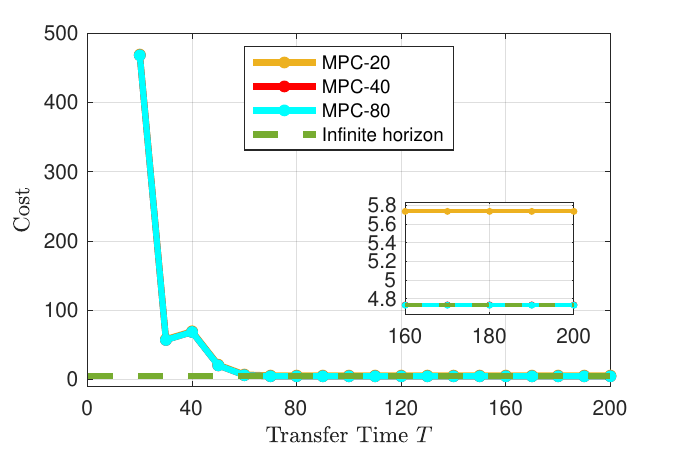}}}
    \sbox3{\subfloat[Car-like - Terminal cost - Case 2]{\includegraphics[width=0.5\linewidth]{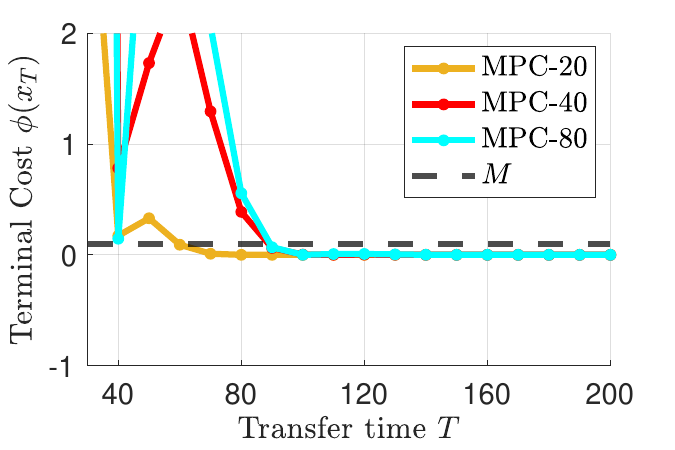}}}
    \centering
    \usebox0\hfil \usebox1\par
  \usebox2\hfil \usebox3\par
    \caption{Comparison of different MPC policies on the car-like robot system. It can be observed that the optimal horizon for MPC depends on the initial state, and choosing a small horizon leads to suboptimal performance.}
    \label{fig:mpc}
\end{figure}
\begin{table}[!htbp]
    \centering
    \begin{tabular}{|c|c|c|c|c|}
        \hline
         \textbf{Control Law} & \textbf{Transfer Time: Case-1} & \textbf{Case-2}  \\
         \hline 
         the
         Ours (Finite horizon) & 80 & 40\\
         \hline
         MPC - 20 & 70 & 60\\
         \hline
         MPC - 40 & 130 & 90\\
         \hline
         MPC - 80 & 130 & 90\\
         \hline
    \end{tabular}
    \caption{Transfer time into the terminal set for each control law used. }
    \label{tab:transfer_times}
\end{table}

\section{Conclusions}
In this paper, we have developed a tractable approach to the approximate solution of nonlinear infinite horizon optimal control problems that is globally asymptotically stabilizing and converges to the true optimal solution in the limit of a vanishing terminal set. We relax the requirement of linear controllability around the origin used in previous work and extend the results to applications involving nonholonomic systems. Empirical results show that the practical convergence occurs in a very short time and differs based on the initial state of the system, justifying the need for a free final time formulation. Future work will involve the incorporation of state and control constraints and the testing of the approach on a suite of nonlinear problems with varying degrees of complexity. We shall also consider the extension of the approach to the problem of optimal nonlinear output feedback control along with a suitable data-based generalization.
\printbibliography
\end{document}